\documentclass[12pt,reqno]{amsart}

\usepackage{amsmath,amsthm,amssymb}
\usepackage{amssymb}
\usepackage{amsmath}
\usepackage[all]{xy}
\usepackage{tikz-cd}
\usepackage{tikz}

\newtheorem{theorem}{Theorem}[section]

\newtheorem{lemma}[theorem]{Lemma}
\newtheorem{corollary}[theorem]{Corollary}

\numberwithin{equation}{section}

\theoremstyle{definition}
\newtheorem{definition}[theorem]{Definition}

\usepackage[all]{xy}

\newcommand{\R}{{\mathbb R}}
\newcommand{\C}{{\mathbb C}}

\newcommand{\Z}{{\mathbb Z}}

\def\dim{{\rm dim}}

\def\Pic{{\rm Pic}}
\def\mod{{\rm mod}}
\def\Aut{{\rm Aut}}

\def\Alb{{\rm Alb}}

\def\Amp{{\rm Amp}}
\def\Nef{{\rm Nef}}

\begin{document}

\title[Torelli theorem for the product moduli spaces]{Torelli theorem for 
the product of moduli spaces of vector bundles over a curve}

\author[I. Biswas]{Indranil Biswas}

\address{Department of Mathematics, Shiv Nadar University, NH91, Tehsil Dadri Greater Noida, Uttar Pradesh 
201314, India}

\email{indranil29@gmail.com, indranil.biswas@snu.edu.in}

\author[A. Dey]{Arijit Dey}
\address{Department of Mathematics, IIT Madras, Chennai, India}
\email{arijitdey@gmail.com}

\author[A. Pahari]{Anubhab Pahari}
\address{Department of Mathematics, IIT Madras, Chennai, India}
\email{anubhabpahari@gmail.com, ma22d012@smail.iitm.ac.in}

\begin{abstract}
We prove a Torelli-type theorem for a product of the moduli spaces of semistable vector bundles
over smooth projective curves of genus $g\,\ge\, 4$. A similar result is proved
for the product of the moduli stacks of semistable vector bundles.
This is proved using a decomposition theorem for the product of the normal projective 
varieties with Picard rank one and discrete Picard group. As an application, we compute the automorphism group of the product of the moduli spaces and moduli stacks.
\end{abstract}



\maketitle

\section{Introduction}

The classical Torelli theorem states that two smooth complex projective curves $X$ and $X'$ are isomorphic if 
their Jacobians $J(X)$ and $J(X')$ are isomorphic as polarized abelian varieties with polarization given by a 
theta divisor. Thus, the geometry of a curve is completely encoded in its 
principally polarized Jacobian. This was generalized to the moduli spaces of vector bundles on curves.

Given a smooth complex projective curve $X$ and a line bundle $L$ on it, let
$M_X(r,L)$ denote the moduli space of semistable vector
bundles of rank $r$ on $X$ with fixed determinant $L$.
Mumford and Newstead \cite{Mumford} and Tyurin \cite{Tyurin1} proved that the isomorphism class of 
$M_X(2,L)$, with ${\rm degree}{L}$ being an odd integer, uniquely determines the isomorphism class of $X$,
in other words, $X\,\cong\, X'$ if $M_X(2,L)\,\cong\, M_{X'}(2,L')$. Subsequently, Tyurin \cite{Tyurin2} and 
Narasimhan and Ramanan \cite{Narasimhan1} proved for the general case; in other words, when \\
($r$ , ${\rm degree}({L})$ ) = 1= ($r$ , ${\rm degree}({L'})$ ), and $M_X(r,L)\,\cong\, M_{X'}(r,L')$, then $X\,\cong\, X'$.

Later, Kouvidakis and Pantev \cite{Kouvidakis} proved the most general
case for moduli spaces of semistable vector bundles over curves of genus greater 
than $2$; they did not assume the coprimality condition on rank and degree. In addition, they also 
computed the automorphism group of $M_X(r,L)$. Using different geometric techniques, Hwang, Ramanan 
\cite{Hwang} and Biswas, G\'omez and M\~unoz \cite{Biswas} proved the Torelli theorem and also
computed the automorphism group of $M_X(r,L)$ for curve $X$ of genus $g\,\ge\, 4$. More recently, Alfaya, Biswas, G\'omez 
and Mukhopadhyay \cite{Alfaya2}  proved the Torelli theorem for Moduli stacks of principal $G$ bundles on a curve 
of genus at least $3$, where $G$ is any complex reductive algebraic group.

The aim here is to prove a Torelli-type theorem for the product of the moduli spaces of semistable 
vector bundles. More precisely, we prove the following (see Theorem \ref{thm:torelli for prod of moduli space}).

\begin{theorem}\label{thm-a}
For $1\, \leq\, i\, \leq\, m$ (respectively, $1\, \leq\, j\, \leq\, n$) let
$X_i$ (respectively, $X_j'$) be an irreducible smooth projective curves with genus at least $4$.
Fix an integer  $r_i\, \geq\, 2$ (respectively, $r_j' \,\ge\, 2$). Let $M_{X_i}(r_i,d_i)$
(respectively $M_{X_j'}(r'_j, d_j')$) be the moduli space of semistable vector bundles over $X_i$
(respectively, $X_j'$) of rank $r_i$ (respectively, $r_j'$) and degree $d_i$ (respectively, $d_j'$).
Then,
\begin{equation*}
\prod_{i=1}^m M_{X_i}(r_i,d_i)\ \cong\ \prod_{j=1}^n M_{X_j'}(r_j',d_j')
\end{equation*}
if and only if the followings four statements hold:
\begin{enumerate}
\item $m\,=\, n$;

\item There is a permutation $\sigma$ of $\{1,\, \cdots,\, m\}$ such that
$$X_i\ \cong\ X'_{\sigma(i)}$$
for all $i\,\in\, \{1,\, \cdots,\, m\}$;

\item $r_i\ =\ r_{\sigma(i)}'$;

\item $d_i\ \equiv\ \pm d'_{\sigma(i)}$\ \ $(\mod\ \; r_i)$
\end{enumerate}
\end{theorem}

A similar result is proved for the moduli stacks of semistable vector bundles $\mathcal{M}_X(r,L)$
(see Theorem \ref{thm:torelli for prod of stacks}).

\begin{theorem}\label{thm-b}
For $1\, \leq\, i\, \leq\, m$ (respectively, $1\, \leq\, j\, \leq\, n$) let
$X_i$ (respectively, $X_j'$) be an irreducible smooth projective curve whose genus is at least $4$.
Fix an integer  $r_i\, \geq\, 2$ (respectively, $r_j' \,\ge\, 2$). Let ${\mathcal M}_{X_i}(r_i,L_i)$
(respectively ${\mathcal M}_{X_j'}(r'_j, L_j')$) be the moduli stack of semistable vector bundles over $X_i$
(respectively, $X_j'$) of rank $r_i$ (respectively, $r_j'$) and fixed determinant $L_i$ (respectively, $L_j'$).
Then
\begin{equation}
\prod_{i=1}^m {\mathcal M}_{X_i}(r_i,L_i)\ \cong\ \prod_{j=1}^n {\mathcal M}_{X_j'}(r_j',L_j') 
\end{equation}
if and only if the following four statements hold:
\begin{enumerate}
\item $m\,=\, n$;

\item There is a permutation $\sigma$ of $\{1,\, \cdots,\, m\}$ such that
$$X_i\ \cong\ X'_{\sigma(i)}$$
for all $i\,\in\, \{1,\, \cdots,\, m\}$;

\item $r_i\ =\ r_{\sigma(i)}'$;

\item ${\rm degree}(L_i)\ \equiv\ \pm {\rm degree}(L'_{\sigma(i)})$\ \ $(\mod\ \; r_i)$
\end{enumerate}
\end{theorem}

The proof of Theorem \ref{thm-a} proceeds as follows. The determinant map 
$\det\,:\,M_X(r,d)\,\longrightarrow\, \Pic^d(X)$ defined by $E\, \longmapsto\, \bigwedge^rE$
is the \textit{Albanese morphism} (see Lemma \ref{lem: Albanese}). 
Therefore, any isomorphism
$$
\prod_{i=1}^m M_{X_i}(r_i,d_i)\ \cong\ \prod_{j=1}^n M_{X_j'}(r_j',d_j')
$$
induces an isomorphism between the Albanese varieties. This produces a commutative diagram (cf. Diagram 
\eqref{diag: Albanese}) relating determinant morphisms on both sides. Restricting to fibers of
the determinant map, we get an 
isomorphism between the product of moduli spaces with fixed determinants (see \eqref{eq: Isom of 
prod of fixed det moduli}), which is an isomorphism between the product of normal projective varieties with 
Picard rank 1 and discrete Picard varieties. Thus, we are naturally led to a general structural statement 
about products of such varieties. This is established in Theorem \ref{thm:classification_fano} in Section 2.

\begin{theorem}
Let $X_1,\, \cdots,\, X_m$ and $Y_1,\, \cdots,\, Y_n$ be normal projective varieties, with
$\rho(X_i)\, =\, 1\,=\, \rho(Y_j)$ and trivial $\Pic^0(X_i)\,=\, 0\,=\, \Pic^0(Y_j)$ for all
$1\, \leq\, i\, \leq\, m$ and $1\, \leq\, j\, \leq\, n$. If there exists an isomorphism
$$ X\ :=\ \prod_{i=1}^m X_i\ \cong\ \prod_{j=1}^n Y_j\ =:\ Y,$$
then $m\,=\,n$ and there is a permutation $\sigma$ of $\{1,\, \cdots, \, m\}$ such that
$X_i \,\cong\, Y_{\sigma(i)}$ for all $1\, \leq\, i\, \leq\, m$.
\end{theorem}

We also compute the automorphism group of the product of the moduli 
spaces and moduli stacks (see Corollary \ref{cor: aut gp of the product}).

\begin{corollary}
Take moduli spaces $\{M_{X_i}(r_i,L_i)\}_{i=1}^m$ that are mutually non-isomorphic, which
means that $M_{X_i}(r_i,L_i)$ is isomorphic to $M_{X_j}(r_j,L_j)$ only if $i\,=\, j$. Then
$$\Aut\Big(\prod_{i=1}^m M_{X_i}(r_i,L_i)^{n_i}\Big)\ \cong\ {\prod_{i=1}^m
\Aut\Big(M_{X_i}(r_i,L_i)\Big)}^{n_i} \rtimes S_{n_i},$$
where $S_{n_i}$ is the group of permutations of $\{1,\, \cdots, \, n_i\}$; the semi-direct product
is given by the action of $S_{n_i}$ on $M_{X_i}(r_i,L_i)^{n_i}$ which simply permutes the factors.

Similarly, $$\Aut\Big(\prod_{i=1}^m {\mathcal M}_{X_i}(r_i,L_i)^{n_i}\Big)\ \cong\ {\prod_{i=1}^m
\Aut\Big({\mathcal M}_{X_i}(r_i,L_i)\Big)}^{n_i} \rtimes S_{n_i}.$$
\end{corollary}

\section{A decomposition theorem for the product of normal projective varieties with Picard rank one and discrete Picard group.}

\subsection{Preliminaries}

Let $X$ be a normal projective variety over $\mathbb{C}$. Denote by $\Pic(X)$ the group of all 
isomorphism classes of line bundles over $X$. Let
$\Pic^{\tau}(X)\, \subset\, \Pic(X)$ be the subgroup consisting of numerically trivial line
bundles, i.e., $N^1(X)$ is the group of all numerically equivalent
classes of Cartier divisors of $X$.
The \textit{N\'eron--Severi} group of $X$ is 
defined as $$N^1(X)\ :=\ \Pic(X)/\Pic^{\tau}(X).$$
A basic fact about $N^1(X)$ is that it is a free abelian group of finite rank. The rank of $N^1(X)$ is known as 
\textit{Picard number} of $X$, which is denoted by $\rho(X)$. Consequently, $\rho(X)$ is the dimension of 
the $\mathbb R$--vector space $N^1(X)_{\mathbb{R}} \,:=\, N^1(X)\otimes_{\Z}\R$.

The well-known Nakai-Moishezon-Kleiman criterion for ampleness says that a Cartier 
divisor $D$ on $X$ is \textit{ample} if and only if $D^{\dim V}\cdot V \,>\, 0$ for
every positive-dimensional irreducible subvariety $V \,\subseteq\, X$ \cite[Theorem 
1.2.23]{Lazarsfield}. A $\R$--Cartier divisor class $D'\,\in\,N^1(X)_{\R}$ is of the 
form
$$D'\ = \ \sum_i a_i D_i,$$
where $a_i\,\in\, \R$ and $D_i\,\in\, N^1(X)$ for all $i$. A divisor class $D'$ is called 
ample if $D'^{\dim V}\cdot V \,>\,0$ for every $V \,\subseteq\, X$, or equivalently, $D'$ 
is ample if it can be expressed as $\sum_i a_i D_i$, where $a_i\,>\,0$ and $D_i$ is an 
ample Cartier divisor for every $i$. A Cartier divisor $D$ is called \textit{nef} if 
$D^{\dim V} \cdot V\, \ge\, 0$ for every $V\, \subseteq \, X$, or equivalently, if $D 
\cdot C \,\ge\, 0$ for all irreducible curves $C\,\subset\, X$ (Kleiman's theorem). A 
$\R$--Cartier divisor class $D'\,\in\, N^1(X)_{\R}$ is called nef if $D'^{\dim V}\cdot V 
\,\ge\, 0$ for every $V \,\subseteq\, X$. It is clear that any ample class is nef.

The convex cones of all ample (respectively, nef) divisor classes is known as the \textit{ample cone}
$\Amp(X)\,\subset\, N^1(X)_{\R}$ (respectively, \textit{nef cone} $\Nef(X)\,\subset\,
N^1(X)_{\R}$). 
Note that, $\Nef(X)\,=\, \overline{\Amp(X)}$ \cite{Kleiman}, \cite[Theorem 1.4.23(i)]{Lazarsfield},
in particular, $\Nef(X)$ is closed.

Next, let us recall the definition of an \textit{extremal ray}.

\begin{definition}
Let $V$ be a finite-dimensional $\R$ vector space and $K\,\subseteq\, V$ a closed convex cone. A
$1$--dimensional subcone $R\,\subseteq\, K$ is called an extremal ray if for all $u,\,v \,\in \,K$ with
$u+v \,\in\, R$ we have $u,\,v\,\in\, R$.
\end{definition}

\begin{lemma}\label{lem: amp div generating ext ray}
Let $X$ be a normal projective variety with $\rho(X)\,=\,1$. Then $\Nef(X)$ is the unique convex
cone of itself. More precisely, $$\Nef(X)\ =\ \R_{\ge 0}\cdot D,$$
where $D$ is any ample divisor on $X$.
\end{lemma}

\begin{proof}
Since $\rho(X)\,=\,1$, it follows that $N^1(X)_{\R}$ is a $1$--dimensional $\R$ vector space.
Moreover, $\Nef(X)$ is a closed convex cone and it is clear from the definition of $\Nef(X)$ that it is
salient, i.e., $$\Nef(X) \cap (-\Nef(X))\ =\ \{0\}.$$ So
$\Nef(X)$ is the unique convex cone of itself. Also,
$$\Nef(X)\ =\ \R_{\ge 0}\cdot D'$$ for any $D'\,\in\, \Amp(X)$, because $\Amp(X)$ is the interior
of $\Nef(X)$ \cite{Kleiman}, \cite[Theorem 1.4.23(ii)]{Lazarsfield}.
\end{proof}

\begin{lemma}[{\cite[Ex. III.12.6]{hartshorne}}]\label{lem:iso of pic gps}
Let $X_1$ and $X_2$ be two normal projective varieties with $\Pic^0(X_1)\,=\,0$. Then there
is a natural isomorphism
$$
\Pic(X_1 \times X_2)\ \cong\ \Pic(X_1) \times \Pic(X_2). 
$$
\end{lemma}

\begin{proof}
Consider the homomorphism $$ \Phi\ :\ \Pic(X_1) \times \Pic(X_2)\ \longrightarrow\ \Pic(X_1 \times X_2),
\ \ \ (L_1,\, L_2)\ \longmapsto\ \pi_1^*L_1 \otimes \pi_2^*L_2,$$
where $\pi_i\, :\, X_1\times X_2\, \longrightarrow\, X_i$, $i\,\in\, \{1,\, 2\}$,
is the natural projection.
It is evident that $\Phi$ is injective; indeed, the restriction of $\pi_1^*L_1 \otimes \pi_2^*L_2$ to
$X_1\times\{y\}\, \subset\, X_1\times X_2$ (respectively, $\{z\}\times X_2\, \subset\, X_1\times X_2$)
is $L_1$ (respectively, $L_2$).

For the surjectivity of $\Phi$, take any $L \,\in\, \Pic(X_1 \times X_2)$, and
consider the map $\psi\,:\, X_2\, \longrightarrow\, \Pic(X_1)$ defined by
$y\, \longmapsto\, L\big\vert_{X_1 \times \{y\}}$. Since $\Pic^0(X_1)$ is trivial, it follows
that $\Pic(X_1)$ is discrete, and hence $\psi$ is a constant map.
Consider the line bundle $L\otimes \pi^*_1 \psi(X_2)^*$ on $X_1\times X_2$;
note that $\psi(X_2)$ is a line bundle on $X_1$. By the Seesaw
theorem (cf. \cite[p.~51, Corollary 6]{Mumford}) $L\otimes \pi^*_1 \psi(X_2)^*$
is of the form $\pi^*_2 L'$ for some
$L'\, \in\, \Pic(X_2)$. Consequently, we have $L\,=\, \pi^*_1 \psi(X_2)\otimes p^*_2 L'$.
Thus $\Phi$ is also surjective.
\end{proof}

\begin{lemma}
Let $X_1$ and $X_2$ be two normal projective varieties such that $\Pic^0(X_1)\,=\, 0\,
=\, \Pic^0(X_2)$. Then
\begin{equation*}
N^1(X_1 \times X_2)\,\ \cong\,\ N^1(X_1)\,\times\, N^1(X_2).
\end{equation*}
\end{lemma}

\begin{proof}
This follows from Lemma \ref{lem:iso of pic gps}
and the definition of the N\'eron-Severi group.
\end{proof}

Since $\Pic^0(X_1\times X_2)\,=\, \Pic^0(X_1) \times \Pic^0(X_2)$ \cite[Proposition 
2.4]{Nobel}, for a product \(X \,=\, \prod_{i=1}^n X_i\) of normal projective varieties 
with trivial $\Pic^0(X_i)$, we have
\begin{equation}\label{eq:neron severi decomp}
    N^1(X)\ \;\cong\ \; \bigoplus_{i=1}^n N^1(X_i), 
\end{equation}
and therefore
\begin{equation}\label{eq:neron severi decomp with R}
    N^1(X)_{\R}\ \;\cong\ \; \bigoplus_{i=1}^n N^1(X_i)_{\R},
\end{equation}

\subsection{Classification}

\begin{lemma}\label{lem:cone}
Let $X \,=\, \prod_{i=1}^n X_i$ be a product of normal projective varieties with $\rho(X_i)
\,=\, 1$ and $\Pic^0(X_i)\,=\, 0$ for all $i$. Then the following two statements hold:
\begin{enumerate}
\item There is a natural isomorphism $\Nef(X)\,\cong\, \bigoplus_{i=1}^n \Nef(X_i)$. 

\item These are the only extremal rays in $\Nef(X)$.
\end{enumerate}
\end{lemma}

\begin{proof}
\textbf{Statement (1):}\,\, For $1\, \leq\, j\, \leq\, n$, let $\pi_j\, :\, \prod_{i=1}^n X_i
\, \longrightarrow\, X_j$ be the natural projection.
Since pullback of a nef divisor is nef and the sum of a finite number of nef
divisors is nef, it is evident that $\bigoplus_{i=1}^n \Nef(X_i) \,\subseteq\, \Nef(X)$, where
$\Nef(X_i)$ is considered as a subgroup of $\Nef(X)$ using pulling by the projection $\pi_i$.

To prove the reverse inclusion, take any nef class $D\,\in\, \Nef(X)\,\subset\, N^1(X)_{\R}$.
For any $1\, \leq\, j\, \leq\, n$, the restriction of $D$ to $X_j\,=\, X_j\times\prod_{i\ne j}\{x_i\}$,
where $x_i\, \in\, X_i$, lies in $\Nef(X_j)$. Using \eqref{eq:neron severi decomp with R} this implies
that $\Nef(X)\, \subset\, \bigoplus_{i=1}^n \Nef(X_i)$. This proves Statement (1).

\textbf{Statement (2):}\,\, Let $R\,=\,\mathbb{R}_{\ge 0}\cdot D$ be any extremal ray of
$\Nef(X)$ and $v\,\in\, R$ a non-zero vector. Since $v\,\in\, \Nef(X)$, using
Statement (1), let
 \begin{equation*}
v\ =\ \sum_{i=1}^n a_i\cdot\pi_{i}^*D_i
 \end{equation*}
with $a_i\, >\, 0$ and $0\, \not=\, D_i\,\in\, \Nef(X_i)$. Since $R$ is an
extremal ray, this implies that each component $a_i\cdot\pi_{i}^*D_i$, lies in $R$, i.e., 
\begin{equation*}
\begin{split}
 & a'_i\cdot D\ =\ a_i\cdot\pi_{i}^*D_i,
\end{split}
\end{equation*}
where $a'_i\, >\, 0$. We have
 \begin{equation*}
      a'_ia_j\pi_{j}^*D_j\ =\ a'_ia'_j D\ =\ a'_ja_i\pi_{i}^*D_i
 \end{equation*}
for all $1\,\le\, i,\,j\,\le\, n$. Therefore,
$\pi_{i}^*D_i$ and $\pi_{j}^*D_j$ are scalar multiple of each other. Which is a contradiction, since they are
linearly independent in $N^1(X)_{\R}$. 
Consequently, we have $v\,=\,a_i\cdot\pi_{i}^*D_i\,\in\, \Nef(X_i)$ for a fixed $i$. This implies
that $R\,\subseteq\, \Nef(X_i)$, and hence $R\,=\,\Nef(X_i)$. This completes the proof.
\end{proof}

The following is the main result of this section.

\begin{theorem}\label{thm:classification_fano}
Let $X_1,\, \cdots,\, X_m$ and $Y_1,\, \cdots,\, Y_n$ be normal projective varieties, with
$\rho(X_i)\, =\, 1\,=\, \rho(Y_j)$ and  $\Pic^0(X_i)\,=\, 0\,=\, \Pic^0(Y_j)$ for all
$1\, \leq\, i\, \leq\, m$ and $1\, \leq\, j\, \leq\, n$. If there exists an isomorphism
$$ X\ :=\ \prod_{i=1}^m X_i\ \cong\ \prod_{j=1}^n Y_j\ =:\ Y,$$
then $m\,=\,n$ and there is a permutation $\sigma$ of $\{1,\, \cdots, \, m\}$ such that
$X_i \,\cong\, Y_{\sigma(i)}$ for all $1\, \leq\, i\, \leq\, m$.
\end{theorem}

\begin{proof}
Fix an isomorphism $$\Phi\ :\ X\ \xrightarrow{\,\,\,\simeq\,\,\,}\ Y.$$ We have an induced isomorphism
\begin{equation}\label{ps}
\Phi^*\ :\ N^1(Y)_{\R}\ \xrightarrow{\,\,\,\simeq\,\,\,}\ N^1(X)_{\R}.
\end{equation}
Therefore, following \eqref{eq:neron severi decomp with R}, $m\,=\,\rho(X)\,=\,\rho(Y)\,=\,n.$

The isomorphism $\Phi^*$ in \eqref{ps} preserves nefness, i.e., $\Phi^*\,:\,\Nef(Y)\,
\xrightarrow{\,\,\,\simeq\,\,\,}\,\Nef(X)$. Moreover, $\Phi^*$ evidently induces a bijection between the
extremal rays of $\Nef(X)$ and $\Nef(Y)$ \cite[Example 1.4.4]{Lazarsfield} . In view of Lemma \ref{lem:cone}, this implies that there is
permutation $\sigma$ of $\{1,\, \cdots, \, m\}$ such that the isomorphism $\Phi^*$ satisfies the
following condition:
\begin{equation}\label{p1}
\R_{\ge 0}\cdot \varpi_{\sigma(i)}^*D'_{\sigma(i)}\ \cong \Nef(Y_{\sigma(i)})\
\xrightarrow[\,\,\,\Phi^*\,\,\,]{\,\,\,\simeq\,\,\,}\ \Nef(X_i)\ \cong\ \R_{\ge 0}\cdot \pi_i^*D_i,
\end{equation}
where $D_i$ and $D'_{\sigma(i)}$ are ample generators of $\Nef(X_i)$ and $\Nef(Y_{\sigma(i)})$
respectively (cf. Lemma \ref{lem: amp div generating ext ray}), while $\pi_i\, :\, X
\, \longrightarrow\, X_i$ and $\varpi_i\, :\, Y \, \longrightarrow\, Y_i$ are the
natural projections for all $1\, \leq\, i\, \leq\, m$. Note that from Lemma \ref{lem:cone}
it follows that $\{\R_{\ge 0}\cdot D_i\}_{i=1}^m$ (respectively, $\{\R_{\ge 0}\cdot D'_i\}_{i=1}^m$) are
the extremal rays in $\Nef(X)$ (respectively, $\Nef(Y)$).

The homomorphism $\Phi^*\,:\, N^1(Y)\,\xrightarrow{\,\,\,\simeq\,\,\,}\, N^1(X)$ induced
by $\Phi$ is an isomorphism of abelian
groups, i.e., an isomorphism of $\Z$--modules. Since $\Phi^*(\varpi_{\sigma(i)}^*D'_{\sigma(i)})\,\in\,
\R_{\ge 0}\cdot \pi_i^*D_i$, there are positive integers $a_i$ and $b_i$ such that  
$$b_i\cdot \Phi^*(\varpi_{\sigma(i)}^*D'_{\sigma(i)})\ =\ a_i\cdot \pi_i^*D_i $$ 
as elements of $\Nef(X)$. This implies that
$b_i\cdot \Phi^*(\varpi_{\sigma(i)}^*D'_{\sigma(i)})$ and $a_i \cdot \pi_i^*D_i$ are numerically equivalent classes
over $X$. We know that, $\Pic^{\tau}(X)/\Pic^0(X)$ is a finite torsion group
\cite[Theorem 9.6.3, Corollary 9.6.17]{Fantechi}. Since, $\Pic^0(X)$ is trivial, there exists a positive integer
$k_i\,\in\, \Z_{\ge 0}$ such that 
\begin{equation}\label{p2}
k_ib_i\cdot \Phi^*(\varpi_{\sigma(i)}^*D'_{\sigma(i)})\ \cong\ k_ia_i\cdot \pi_i^*D_i
\end{equation}
as Cartier divisors over $X$. Since $D_i$ and $D'_{\sigma(i)}$ are ample, we may choose
\begin{equation}\label{p3}
n_i\,\ >>\,\ 0
\end{equation}
such that $n_i\cdot D_i$ and $n_i\cdot D'_{\sigma(i)}$ are very ample over $X$ and $Y$
respectively. So from \eqref{p2} it follows that
\begin{equation}\label{eq: iso of cartier div}
n_i k_ib_i\cdot \Phi^*(\varpi_{\sigma(i)}^*D'_{\sigma(i)})\ \cong\ n_i k_i m_i\cdot \pi_i^*D_i
\end{equation}
as Cartier divisors over $X$, and the divisors in \eqref{eq: iso of cartier div} are very ample.

For notational convenience, denote by let $L_i$ (respectively, $L'_{\sigma(i)}$) the line bundle
over $X_i$ (respectively, $Y_i$) corresponding to the divisor $n_i k_i a_i\cdot D_i$
(respectively, $n_i k_ib_i\cdot D'_{\sigma(i)}$). Using this notation, the
isomorphism in \eqref{eq: iso of cartier div} can be rewritten as 
\begin{equation}\label{eq: iso of line bundles}
\Phi^*\varpi_{\sigma(i)}^*L'_{\sigma(i)}\ \, \cong\ \, \pi_i^* L_i.
\end{equation}

Note that the line bundles $L_i$ and $L'_{\sigma(i)}$ are very ample; see \eqref{p3}.
In particular, the line bundles $\pi_i^*L_i$ and $\varpi_i^*L'_i$ are basepoint free. Also, global
sections of $\pi^*_iL_i$ are exactly the pullbacks of sections of $L_i$, i.e.,
$H^0(X,\, \pi_i^*L_i)\,\cong\, H^0(X_i,\, L_i).$ So, the associated morphism
$$\Psi_{\pi^*L_i}\ :\ X\ \longrightarrow\ {\mathbb P}(H^0(X,\, \pi^*_iL_i))\ =\ {\mathbb P}(H^0(X_i,\, L_i))$$
depends only on the $i$--th factor, i.e.,
$\Psi_{\pi^*L_i}$ fits in the diagram 
\begin{center}
     \begin{tikzcd}[column sep=large, row sep=large]
X \ar[d, "\pi_i"', rightarrow] \ar[r, "\Psi_{\pi_i^*L_i}"] & {\mathbb P}(H^0(X,\, \pi^*_iL_i))\,
=\, {\mathbb P}(H^0(X_i,\, L_i)) \\
X_i \ar[ru, description, "\Psi_{L_i} "']
\end{tikzcd}
 \end{center}  
where $\Psi_{L_i}$ is a closed embedding since $L_i$ is very ample. Therefore, the image of
$\Psi_{\pi_i^*L_i}$ is isomorphic to $X_i$. 

Similarly, since $L'_{\sigma(i)}$ is very ample and $N=
h^0(Y, \pi_{\sigma(i)}^*L'_{\sigma(i)})
$(cf. equation \eqref{eq: iso of line bundles}), we have a commutative diagram 
\begin{center}
     \begin{tikzcd}[column sep=large, row sep=large]
Y \ar[d, "\varpi_{\sigma(i)}"', rightarrow] \ar[r, "\Psi_{\varpi_{\sigma(i)}^*L'_{\sigma(i)}}"] &
{\mathbb P}(H^0(Y,\,\varpi^*_{\sigma(i)}L'_{\sigma(i)}))\,=\,{\mathbb P}(H^0(Y_{\sigma(i)},\,L'_{\sigma(i)}))\\
Y_{\sigma(i)} \ar[ru, description, "\Psi_{L'_{\sigma(i)}} "']
\end{tikzcd}
 \end{center}
where $\Psi_{L'_{\sigma(i)}}$ is a closed embedding. Consequently, the image of
$\Psi_{\varpi_{\sigma(i)}^*L'_{\sigma(i)}}$ is isomorphic to $Y_{\sigma(i)}$.

The isomorphism in \eqref{eq: iso of line bundles} produces an isomorphism of global sections
\begin{equation}\label{q1}
H^0(X,\, \Phi^*\varpi^*_{\sigma(i)}L'_{\sigma(i)})\ \, \cong\ \, H^0(X,\, \pi_i^* L_i).
\end{equation}
But $H^0(X,\, \Phi^*\varpi^*_{\sigma(i)}L'_{\sigma(i)})\,=\, H^0(Y,\, 
\varpi^*_{\sigma(i)}L'_{\sigma(i)})$. So \eqref{q1} gives an isomorphism
\begin{equation}\label{q2}
H^0(Y,\, \varpi^*_{\sigma(i)}L'_{\sigma(i)}) \ \, \cong\ \, H^0(X,\, \pi_i^* L_i).
\end{equation}
Hence we have the commutative diagram
\[
\begin{tikzpicture}[>=Stealth]

\node (X)  at (-3,3) {$X$};
\node (Y)  at ( 3,3) {$Y$};

\node (Xi) at (-5,0) {$X_i$};
\node (P)    at ( 0,0) {$\mathbb{P}({\mathcal W})$};
\node (Ysig)   at ( 5,0) {$Y_{\sigma(i)}$};

\draw[->] (X) -- node[above] {$\Phi$} (Y);

\draw[->] (X) -- node[left,pos=0.6] {$\pi_i$} (Xi);

\draw[->] (Xi) -- node[above,pos=0.5] {$\Psi_{L_i}$} (P);

\draw[->] (X) -- 
node[sloped,above,pos=0.6] 
{$\Psi_{\pi^*_iL_i}$} 
(P);

\draw[->] (Y) -- node[right,pos=0.6] {$\varpi_i$} (Ysig);

\draw[->] (Ysig) -- node[above,pos=0.5] {$\Psi_{L'_{\sigma(i)}}$} (P);

\draw[->] (Y) -- 
node[sloped,above,pos=0.6] 
{$\Psi_{\varpi^*_{\sigma(i)}L_{\sigma(i)}}$} 
(P);

\end{tikzpicture}
\]
where ${\mathcal W}\,=\, H^0(X,\, \pi_i^* L_i)\,=\, H^0(Y,\, \varpi^*_{\sigma(i)}L'_{\sigma(i)})$ (see
\eqref{q2}).

It is clear that $\text{Im}(\Psi_{L_i})\,=\,\text{Im}(\Psi_{L'_{\sigma(i)}})$. Moreover, both
$$\Psi_{L'_{\sigma(i)}}\,:\, \text{Im}(\Psi_{L'_{\sigma(i)}})\,\longrightarrow \,Y_{\sigma(i)}\,\ \
\text{ and }\,\ \ \Psi_{L_i}\,:\, X_i \,\longrightarrow\, \text{Im}(\Psi_{L_i})$$ are isomorphisms.
Therefore, the map $$ \Psi_{L'_{\sigma(i)}}^{-1} \circ \Psi_{L_i}\ :\ X_i\ \longrightarrow\ Y_{\sigma(i)}$$
is an isomorphism. This completes the proof.
\end{proof}

\section{Torelli Theorem for Product of moduli of vector bundles}

Let $M_X(r,d)$ be the moduli space of semi-stable vector bundles of rank $r$ and degree $d$ over a smooth 
projective curve $X$ over $\C$.

\begin{lemma}\label{lem: Albanese}
    The determinant morphism
    \begin{equation*}
        \det\ :\ M_X(r,d)\ \longrightarrow\ \Pic^d(X)
    \end{equation*}
    is an Albanese morphism.
\end{lemma}

\begin{proof}
For any $L \in Pic^d(X)$, $\det^{-1}(L)\, =\, M_X(r,\,L)$, which is a unirational variety. The Albanese  of an unirational variety is trivial, since there is no nonconstant map from a projective space to an abelian variety. 
 %Hence $\Alb(M_X(r,\,L))\,=\,0$, Hence by the universal property of Albanese morphism  $\ell$ is constant on the fibers of $\det$.
      
%By \cite[Section 2.2]{Hein} we have the Albanese
%variety $\Alb(M_X(r,\,L))\,\,=\,\, \Pic^0(\Pic^0(M_X(r,\,L)))\,=\,0$, because $\Pic^0(M_X(r,\,L))\,=\,0$ ~\cite{Drezet}.
Now, let $A$ be an abelian variety, and let $\ell\,:\, M_X(r,d)\, \longrightarrow\, A$ be a morphism. We denote the 
restriction map by, $\ell_{|M_X(r,L)}\,:\, M_X(r,L)\,\longrightarrow\, A$. By the universal property of Albanese,
there is a unique morphism $\ell'\,:\, \Alb(M_X(r,\,L))\,\longrightarrow\, A$ such that $\ell'\circ \Alb
\, =\, \ell_{|M_X(r,L)}$, where $\Alb\,:\, M_X(r,\,L)\,\longrightarrow\, \Alb(M_X(r,\,L))$ is the Albanese morphism.
Since $\Alb(M_X(r,\, L))\,=\,0$, this implies $\ell$ is constant on the fibers of $\det$.

Note that $\det$ is a surjective submersion and has smooth fibers with 
dimension equals to $\dim(M_X(r,\,d))-\dim \Pic^d(X)\,=\, r^2(g-1)+1-g\,=\,(g-1)(r^2-1)$,
and hence $\det$ is smooth morphism; in particular, the morphism $\det$ is 
flat. So, following \cite[Theorem 3.12]{Javanpeykar}, there exists a unique morphism $\ell''\,:\, 
\Pic^d(X)\,\longrightarrow \,A$ such that the following diagram commutes:
\begin{equation}\label{di}
\begin{tikzcd}[column sep=large, row sep=large]
M_X(r,\,d) \ar[d, "\det"', rightarrow] \ar[r, "\ell"] & A \\
\Pic^d(X) \ar[ru, description, "\ell'' "']
\end{tikzcd}
\end{equation} 
Since $\Pic^d(X)\,\cong\, \Pic^0(X)$, it follows that $\Pic^d(X)$ is naturally an abelian variety. Therefore,
from \eqref{di} it is clear that $\Pic^d(X)$, along with the $\det$ map, is the Albanese variety of $M_X(r,\;d)$. 
\end{proof}

Let us recall the well-known Torelli theorem for the moduli spaces \cite{Hwang}, \cite{Kouvidakis}, 
\cite{Mumford}, \cite{Narasimhan1}, \cite{Narasimhan2}, \cite[Corollary 2.12]{Alfaya1}.

\begin{theorem}\label{thm: torelli for moduli}
Let $X_1$ and $X_2$ be irreducible smooth projective curves of genus $g\,\ge\, 4$; fix integers $r_1,\,
r_2\,\ge\, 2$. Let $L_1$ and $L_2$ be line bundles over $X_1$ and $X_2$ respectively. For $j\,=\, 1,\, 2$,
let $M_{X_j}(r_j,L_j)$ denote the moduli space of semistable vector bundles on $X_j$ of rank $r_j$ and
determinant $L_j$.
Then the following  two statements are equivalent:
\begin{enumerate}
\item There is an isomorphism between the moduli spaces
$$M_{X_1}(r_1,L_1)\ \, \cong\ \, M_{X_2}(r_2,L_2).$$
 
\item $X_1\,\cong\, X_2$, $r_1\,=\,r_2$ and $\deg(L_1) \,\equiv\, \pm \deg(L_2)$\ $(\mod \;\; r)$.
\end{enumerate}
\end{theorem}

We prove the following theorem.

\begin{theorem}\label{thm:torelli for prod of moduli space}
For $1\, \leq\, i\, \leq\, m$ (respectively, $1\, \leq\, j\, \leq\, n$) let
$X_i$ (respectively, $X_j'$) be an irreducible smooth projective curves with genus at least $4$.
Fix an integer  $r_i\, \geq\, 2$ (respectively, $r_j' \,\ge\, 2$). Let $M_{X_i}(r_i,d_i)$
(respectively $M_{X_j'}(r'_j, d_j')$) be the moduli space of semistable vector bundles over $X_i$
(respectively, $X_j'$) of rank $r_i$ (respectively, $r_j'$) and degree $d_i$ (respectively, $d_j'$).
Then,
\begin{equation*}
\prod_{i=1}^m M_{X_i}(r_i,d_i)\ \cong\ \prod_{j=1}^n M_{X_j'}(r_j',d_j') 
\end{equation*}
if and only if the followings four statements hold:
\begin{enumerate}
\item $m\,=\, n$;

\item There is a permutation $\sigma$ of $\{1,\, \cdots,\, m\}$ such that
$$X_i\ \cong\ X'_{\sigma(i)}$$
for all $i\,\in\, \{1,\, \cdots,\, m\}$;

\item $r_i\ =\ r_{\sigma(i)}'$;

\item $d_i\ \equiv\ \pm d'_{\sigma(i)}$\ \ $(\mod\ \; r_i)$
\end{enumerate}
\end{theorem}

\begin{proof}
If the above four statements hold, then it is evident that
$\prod_{i=1}^m M_{X_i}(r_i,d_i)\, \cong\, \prod_{j=1}^n M_{X_j'}(r_j',d_j')$.
We will prove the converse.

Fix an isomorphism
$$
\Phi\ :\ \prod_{i=1}^m M_{X_i}(r_i,d_i)\ \xrightarrow{\,\,\,\simeq\,\,\,}\
\prod_{j=1}^m M_{X_j'}(r_j',d_j').
$$
We will prove that the four statements in the theorem hold.

Note that following \cite[Proposition 2.4]{Nobel} and Lemma \ref{lem: Albanese},
\begin{equation*}
\Alb\Big(\prod_{i=1}^m M_{X_i}(r_i,d_i) \Big)\ \cong\ \prod_{i=1}^m \Alb (M_{X_i}(r_i,d_i))\ \cong
\ \prod_{i=1}^m \Pic^{d_i}(X_i)
\end{equation*}
with the Albanese morphism being the product of coordinate-wise Albanese morphisms. Similarly,
$$\Alb\Big(\prod_{j=1}^n M_{X_j'}(r_j',d_j') \Big)\ \cong\ \prod_{j=1}^n \Pic^{d_j'}(X_j').$$ Therefore, from
the universal property of the Albanese there is the following commutative diagram
\begin{equation}\label{diag: Albanese}
\begin{array}{ccc}
\prod_{i=1}^m M_{X_i}(r_i,d_i) & \xrightarrow[\Phi]{\,\,\,\simeq\,\,\,} & \prod_{j=1}^n M_{X_j'}(r_j',d_j') \\
\Big\downarrow \det & & \Big\downarrow \det \\
\prod_{i=1}^m \Pic^{d_i}(X_i) & \xrightarrow[\phi]{\,\,\,\simeq\,\,\,} & \prod_{j=1}^n \Pic^{d_j'}(X_j')
\end{array}
\end{equation}
where the isomorphism $\phi$ is induced by the isomorphism $\Phi$.

Take $\{L_1,\,\cdots,\,L_m\}\,\in\, \prod_{i=1}^m \Pic^{d_i}(X_i)$, and let
$$\phi(\{L_1,\,\cdots ,\, L_m\})\ :=\ \{L'_1,\,\cdots,\, L'_n\}\ \in\ \prod_{j=1}^n \Pic^{d_j'}(X_j'),$$
where $\phi$ is the morphism in \eqref{diag: Albanese}.
Restricting  $\Phi$ (see \eqref{diag: Albanese}) to the fibers of $\det$ we have the following:
\begin{equation}\label{eq: Isom of prod of fixed det moduli}
\Phi\ :\ \prod_{i=1}^m M_{X_i}(r_i,L_i)\ \xrightarrow{\,\,\,\simeq\,\,\,}\ \prod_{j=1}^n M_{X_j'}(r_j',L'_j).
\end{equation}

In view of the isomorphism in \eqref{eq: Isom of prod of fixed det moduli}, from Theorem
\ref{thm:classification_fano} and the fact that ${\rm Pic}(M_{X_i}(r_i,L_i))\,=\,
{\mathbb Z}\,=\, M_{X_j'}(r_j',L'_j)$ and  $\Pic^0(M_{X_i}(r_i,L_i))\,=\, 0\,=\, \Pic^0(M_{X_j'}(r_j',L'_j))$ for all
$1\, \leq\, i\, \leq\, m$ and $1\, \leq\, j\, \leq\, n$ (see \cite[p.~55, Theorem B]{Drezet})
the following is deduced:
\begin{enumerate}
\item $n\,=\,m$, and

\item there is a permutation $\sigma$ of $\{1,\, \cdots,\, m\}$ such that
$$M_{X_i}(r_i,L_i)\ \cong\ M_{X_{\sigma(i)}}(r_j',L'_{\sigma(i)})$$
for all $1\, \leq\, i\, \leq\, m$.
\end{enumerate}
Now from Theorem \ref{thm: torelli for moduli} it follows that $X_i\,\cong\, X_{\sigma(i)}'$,
$r_i\,=\,r_{\sigma(i)}'$  and $d_i\,\equiv\,\pm d_{\sigma(i)}' \quad (\mod \;\; r_i)$.
This completes the proof.
\end{proof}

Theorem \ref{thm:torelli for prod of moduli space} can be extended to the case of moduli stacks.
For a line bundle $L$ over a smooth projective curve $X$,
let $\mathcal{M}_X(r, L)$ be the moduli stack of semistable vector bundles over $X$ of rank $r$
and fixed determinant $L$.

\begin{theorem}\label{thm:torelli for prod of stacks}
For $1\, \leq\, i\, \leq\, m$ (respectively, $1\, \leq\, j\, \leq\, n$) let
$X_i$ (respectively, $X_j'$) be an irreducible smooth complex projective curves with genus at least $4$.
Fix an integer  $r_i\, \geq\, 2$ (respectively, $r_j' \,\ge\, 2$). Let ${\mathcal M}_{X_i}(r_i,L_i)$
(respectively ${\mathcal M}_{X_j'}(r'_j, L_j')$) be the moduli stack of semistable vector bundles over $X_i$
(respectively, $X_j'$) of rank $r_i$ (respectively, $r_j'$) and fixed determinant $L_i$ (respectively, $L_j'$).
Then,
\begin{equation}\label{j1}
\prod_{i=1}^m {\mathcal M}_{X_i}(r_i,L_i)\ \cong\ \prod_{j=1}^n {\mathcal M}_{X_j'}(r_j',L_j') 
\end{equation}
if and only if the followings four statements hold:
\begin{enumerate}
\item $m\,=\, n$;

\item There is a permutation $\sigma$ of $\{1,\, \cdots,\, m\}$ such that
$$X_i\ \cong\ X'_{\sigma(i)}$$
for all $i\,\in\, \{1,\, \cdots,\, m\}$;

\item $r_i\ =\ r_{\sigma(i)}'$;

\item ${\rm degree}(L_i)\ \equiv\ \pm {\rm degree}(L'_{\sigma(i)})$\ \ $(\mod\ \; r_i)$
\end{enumerate}
\end{theorem}

\begin{proof}
If the four statements in the theorem hold, then it can be shown that
$$\prod_{i=1}^m {\mathcal M}_{X_i}(r_i,L_i)\ \cong\ \prod_{j=1}^n {\mathcal M}_{X_j'}(r_j',L_j').$$
Indeed, this follows from \cite[Theorem 3.6]{Alfaya2} and then constructing a coordinate-wise morphism.

To prove the converse, assume that there is an isomorphism as in \eqref{j1}.

Following \cite[Lemma 4.15 and Theorem 6.6]{Alper} we see that
$\prod_{i=1}^m \mathcal{M}_{X_i}(r_i,L_i)$ (respectively, $\prod_{j=1}^n \mathcal{M}_{X_j'}(r_j',L_j')$)
is corepresented by $\prod_{i=1}^m M_{X_i}(r_i,L_i)$ (respectively, $\prod_{j=1}^m M_{X_j'}(r_j',L_j')$.
Therefore, we have a commutative diagram
   \begin{equation}\label{diag: corep of moduli stacks}
\begin{array}{ccc}
  \prod_{i=1}^m \mathcal{M}_{X_i}(r_i,L_i) & \xrightarrow{\,\,\,\simeq\,\,\,} & \prod_{j=1}^n \mathcal{M}_{X_j'}(r_j',L_j') \\
\Big\downarrow & & \Big\downarrow \\
\prod_{i=1}^m M_{X_i}(r_i,L_i) & \xrightarrow{\,\,\,\simeq\,\,\,} & \prod_{j=1}^n M_{X_j'}(r_j',L_j')
\end{array}
\end{equation}
The isomorphism at the bottom is induced by the isomorphism at the top, by using the universal property
of good moduli spaces. In view of the isomorphism
$$
\prod_{i=1}^m M_{X_i}(r_i,L_i) \ \xrightarrow{\,\,\,\simeq\,\,\,} \ \prod_{j=1}^n M_{X_j'}(r_j',L_j')
$$
in \eqref{diag: corep of moduli stacks} it follows from Theorem \ref{thm:torelli for prod of moduli space}
that the four statements in the theorem hold. This completes the proof.
\end{proof}

The automorphism group of $M_X(r,L)$ is denoted by $\Aut(M_X(r,L))$; see \cite{Kouvidakis}, \cite{Hwang}, 
\cite{Biswas}, \cite{Alfaya2} for results on $\Aut(M_X(r,L))$.

\begin{corollary}\label{cor: aut gp of the product}
Take moduli spaces $\{M_{X_i}(r_i,L_i)\}_{i=1}^m$ that are mutually non-isomorphic, which
means that $M_{X_i}(r_i,L_i)$ is isomorphic to $M_{X_j}(r_j,L_j)$ only if $i\,=\, j$. Then
$$\Aut\Big(\prod_{i=1}^m M_{X_i}(r_i,L_i)^{n_i}\Big)\ \cong\ {\prod_{i=1}^m
\Aut\Big(M_{X_i}(r_i,L_i)\Big)}^{n_i} \rtimes S_{n_i},$$
where $S_{n_i}$ is the group of permutations of $\{1,\, \cdots, \, n_i\}$; the semi-direct product
is given by the action of $S_{n_i}$ on $M_{X_i}(r_i,L_i)^{n_i}$ which simply permutes the factors.

Similarly, $$\Aut\Big(\prod_{i=1}^m {\mathcal M}_{X_i}(r_i,L_i)^{n_i}\Big)\ \cong\ {\prod_{i=1}^m
\Aut\Big({\mathcal M}_{X_i}(r_i,L_i)\Big)}^{n_i} \rtimes S_{n_i}.$$
\end{corollary}

\begin{proof}
The group $\Aut\Big(\prod_{i=1}^m M_{X_i}(r_i,L_i)^{n_i}\Big)$ acts on the nef cone of
$\prod_{i=1}^m M_{X_i}(r_i,L_i)^{n_i}$, and this action preserves the extremal rays in
the nef cone of $\prod_{i=1}^m M_{X_i}(r_i,L_i)^{n_i}$. Fix an ample divisor $D_i$
on $M_{X_i}(r_i,L_i)$. As noted in the proof of
Theorem \ref{thm:classification_fano}, the extremal $\Nef(\prod_{i=1}^m M_{X_i}(r_i,L_i)^{n_i})$
is given by $\{\R_{\ge 0}\cdot D_{i,j}\}_{j=1}^{n_i}$, where $D_{i,j}$ is the divisor $D_j$
in the $j$-th factor of $\prod_{i=1}^m M_{X_i}(r_i,L_i)^{n_i}$. Consequently, we get a homomorphism
$$
\Aut(\prod_{i=1}^m M_{X_i}(r_i,L_i)^{n_i})\ \longrightarrow\ S_{n_i}.
$$
The kernel of this homomorphism is $\prod_{i=1}^m
\Aut\Big(M_{X_i}(r_i,L_i)\Big)^{n_i}$. Now the first statement (on moduli spaces) follows from this.

The second statement (on moduli stacks) follows from \eqref{diag: corep of moduli stacks} and the
first statement.
\end{proof}

\section*{Acknowledgements}

The third-named author wishes to thank the Department of Mathematics, IIT Madras, for excellent working 
conditions, and the Prime Minister's Research Fellowship (PMRF, ID: 2503482) for their financial support. The 
authors want to thank Dr. Arijit Mukherjee for many helpful discussions.

\end{document}